\documentclass[12pt]{amsart}
\usepackage{pinlabel}

\usepackage{graphicx}
\usepackage{multirow}
\usepackage[english]{babel}
\usepackage{amsrefs}
\graphicspath{ {images/} }
\usepackage[T1]{fontenc}
\usepackage{txfonts}
\usepackage{times}
\usepackage[all]{xy}
\usepackage{subfig}
\usepackage{tikz}
\usetikzlibrary{shapes,arrows,shadows}
\usetikzlibrary{decorations.markings}
\usepackage{color}
\usepackage[normalem]{ulem}
\usepackage{hyperref}
\usepackage{wrapfig}
\usetikzlibrary{arrows}
\usepackage{pb-diagram}
\usepackage{verbatim}
\usepackage{float}
\usepackage{graphpap}
\usepackage{pst-all}
\usepackage{pstcol}
\usepackage{pstricks}
\usetikzlibrary{arrows}
\usetikzlibrary{calc}

\newtheorem{theorem}{Theorem}[section]
\newtheorem{lemma}{Lemma}[section]

\newtheorem{corollary}{Corollary}[section]
\newtheorem{definition}{Definition}[section]

\newtheorem{remark}{Remark}[section]

\newcommand{\Psequence}{\mathbb{R}_{+}^{\mathbb{N}_0}}
\newcommand{\mayor}{>}
\newcommand{\menor}{<}

\begin{document}

\title[ ]{Parabolicity of zero-twist tight flute surfaces and uniformization of the Loch Ness monster}

\author[ ]{John A. Arredondo, Israel Morales, Camilo Ram\'irez Maluendas}

\address{John A. Arredondo\\\newline Departamento de Matem\'aticas, Fundaci\'on Universitaria Konrad Lorenz, Bogot\'a, Colombia.}
\email{alexander.arredondo@konradlorenz.edu.co}

\address{Israel Morales\\\newline Instituto de Matem\'aticas UNAM Unidad Oaxaca, Oaxaca de Ju\'arez, Oaxaca, M\'exico.}
\email{fast.imj@gmail.com}

\address{Camilo Ram\'irez Maluendas\\\newline Departamento de Matem\'aticas y Estad\'istica, Universidad Nacional de Colombia, Sede Manizales, Manizales 170004, Colombia.}
\email{camramirezma@unal.edu.co}



\begin{abstract}
We study the zero-twist flute surface and we associate to each one of them a sequence of positive real numbers $\mathbf{x}=(x_{n})_{n\in\mathbb{N}_{0}}$, with a torsion-free Fuchsian group $\Gamma_{\mathbf{x}}$ such that the convex core  of $\mathbb{H}^2/\Gamma_{\mathbf{x}}$ is isometric to a zero-twist tight flute surface $S_{\mathbf{x}}$. Moreover, we prove that the Fuchsian group $\Gamma_{\mathbf{x}}$ is of the first kind if and only if the series $\sum x_{n}$ diverges. As consequence of the recent work of  Basmajian, Hakobian and {\v{S}}ari{\'c}, we obtain that the zero-twist flute surface $S_{\mathbf{x}}$ is of parabolic type if and only $\sum x_{n}$ diverges.  In addition, we present an uncountable family of hyperbolic surfaces homeomorphic to the Loch Ness Monster. More precisely, we associate to each sequence $\mathbf{y}=(y_{n})_{n\in\mathbb{Z}}$, where $y_{n}=(a_{n},b_{n},c_{n},d_{n},e_{n})\in \mathbb{R}^5$ and $a_n\leq b_n \leq c_n \leq d_n \leq e_n \leq a_{n+1}$, a Fuchsian group $G_{\mathbf{y}}$ such that $\mathbb{H}^2/G_{\mathbf{y}}$ is homeomorphic to the Loch Ness Monster. 
\end{abstract}

\maketitle


\section*{Introduction}\label{sec:introduction}

A fundamental open question in the classification of Riemann surfaces, known as the \textit{type problem} is: 

{\centerline{whether a Riemann surface S supports a Green's function?}}

\noindent 
In the context of simply connected Riemann surfaces we have that  the Poincar\'e disc $\Delta$ admits Green's functions, while the Riemann sphere $\widehat{\mathbb{C}}$ and the complex plane $\mathbb{C}$ do not. From this,  recall that a Riemann surface $S$ is \textbf{elliptic} if and only if $S$ is compact (equivalently, closed). On the other hand, an open Riemann surface $S$ is said to be \textbf{parabolic}, if $S$ does not carry a negative non-constant subharmonic function. All open Riemann surface which are not parabolic will be called \textbf{hyperbolic} (see \cite{Bear1}*{p.164}, \cite{Makoto}*{Section 2}).
It is well known that the type problem becomes equivalent  to answering the question

{\centerline{which open Riemann surfaces are of parabolic type?}}

\noindent The interest in solving this problem has captured the attention of the mathematical community for almost one hundred years. Some of the first partial results to this question, in modern language can be read in the theorems of section 6 in \cite{AhlSar}*{p. 204}. After these results, numerous known characterizations of parabolic surfaces, that are equivalent to the type problem have been produced from potential theory, theory of functions, dynamics and geometry of surfaces, among others. This is a short list of known characterizations of Riemann surfaces of parabolic type: if the Riemann surface $S$ is the quotient of the Poincar\'e unit disk $\Delta$ by a Fuchsian group $\Gamma$, then $S$ is parabolic if and only if 
\begin{itemize}
	\item[$\bullet$] The series $\sum\limits_{\gamma\in \Gamma}\left(\frac{1-\Vert \gamma(\textbf{0})\Vert}{1+\Vert \gamma(\textbf{0})\Vert}\right)$ diverges \cite{Nicho}*{Theorem 5.2.1},
	\item[$\bullet$]  The geodesic  flow on the unit tangent bundle of $S$ is ergodic \cite{Nicho}*{Theorem 8.3.4};
	\item[$\bullet$] The Mostow rigidity holds for $\Gamma$ \cite{Agard}, \cite{AstZin1990}, \cite{Tukia};
	\item[$\bullet$] The group $\Gamma$ has the Bowen's property \cite{AstZin1990}, \cite{Bishop}.
\end{itemize}

\medskip
\noindent  {\bf Parabolicity of zero-twist tight flute surfaces.} Open Riemann surfaces are classified, up to homeomorphisms, by their genus $g(S)\in \mathbb{N}_{0}\cup\{\infty\}$, and a pair of nested topological spaces ${\rm Ends}_{\infty}(S)\subseteq {\rm Ends}(S)$ homeomorphic to a pair of nested closed subsets of the Cantor space.
In this context, the \emph{flute surfaces} are the unique infinite-type Riemann surfaces, up to homeomorphism, of genus zero and with space of ends homeomorphic to the ordinal number $\omega + 1$ (\S \ref{Subsec:TopologicalSurfaces}). The classof flute surfaces is one of the simplest families where we can ask about the parabolicity problem. However, even in this class the problem of parabolicity is widely open (see \cite{BasHakSa}, \cite{MR}). A better manage family in the class of Riemann flute surfaces are the so-called \emph{tight flute surfaces}. Such surfaces are hyperbolic surfaces obtained by starting with a geodesic pair of pants $P_0$ with two punctures and one boundary geodesic and consecutively gluing geodesic pair of pants $P_n$ ($n\geq 1$) with one cusp and two boundary geodesics in an infinite chain (\S \ref{definition:tight_flute_surface}). A tight flute surface is determined by its \emph{Fenchel-Nielsen parameters} (\S \ref{Subsec:FenchelNielsen}) $(\{l_n,t_n\})_{n\in\mathbb{N}}$ where $l_n$ and $t_n$ are the length and twist parameter of the boundary closed geodesic $\alpha_n$ of the surface obtained after gluing $n$ pair of pants. This surface is denoted by $S=S(\{l_{n},t_n\})_{n\in \mathbb{N}}$. If $t_n=0$ for all $n\in \mathbb{N}$, then we say that $S=S(\{l_{n},0\})_{n\in\mathbb{N}}$ is a \emph{zero-twist flute surface}. 

In this article, we aim to present a new way of describing zero-twist flute surfaces; this description differs from the Fenchel-Nielsen parameters. Our main contribution is as follows:

\begin{theorem}\label{Teo:Parametrization-ZTFS}
 For each zero-twist tight flute surface $S=S(\{l_{n},0\})_{n\in\mathbb{N}_{0}}$,  there exists a unique sequence ${\mathbf{x}}=(x_n)_{n\in \mathbb{N}_{0}}$ of positive real numbers, which defines the Fuchsian group $\Gamma_{\mathbf{x}}$ generated by
      \begin{equation}\label{eq:maps_g}
    g_n(z):=\dfrac{\left(1+\dfrac{2s_{n-1}}{s_n-s_{n-1}}\right)z-2s_{n-1}\left(1+\dfrac{s_{n-1}}{s_n-s_{n-1}}\right)}{-\dfrac{2}{s_n-s_{n-1}}z+\left(1+\dfrac{2s_{n-1}}{s_n-s_{n-1}}\right)} \quad  \in PSL(2,\mathbb{R}),
    \end{equation}
with  $n\in\mathbb{N}_{0}$ and $s_{n}=\sum\limits_{i=0}^{n}x_{i}$,  such that the convex core of the surface $\mathbb{H}^{2}/\Gamma_{\mathbf{x}}$ is isometric to $S$. Moreover, $\Gamma_{\mathbf{x}}$ is of the first kind if and only if the series $\sum x_n$ diverges. In this case, $\mathbb{H}^{2}/\Gamma_{\mathbf{x}}$ coincides with $S$ and it is geodesically complete.    
\end{theorem}

    Our construction above is inspired on A. Basmajian work in \cite{Bas93}, where the author present a realization theorem showing that $(\{l_n, t_n\})_{n\in\mathbb{N}}$ are the Fenchel-Nielsen coordinates of a tight flute surface if and only if $r_n\mayor \ln 2$ for all $n$ odd, and $r_n\menor \ln 2$ for all $n$ even, where $r_n:=\sum_{i=1}^n(-1)^{i+1}l_i$, see \cite{Bas93}*{Theorem 2}. Observe that these conditions does not depend on twist parameters. In contrast with our construction of a zero-twist tight flute surface in terms of the sequence $\mathbf{x}=(x_n)_{n\in \mathbb{N}_0} \in \Psequence$, we have that it does not depend of any additional condition on the sequence. Additionally, we can recover the Fenchel-Nielsen parameters of $S=S(\{l_n,0\})_{n\in\mathbb{N}_{0}}$ using 
    
    \begin{equation}\label{Eq:LengthParameterZTFS}
    l_n=\ln\left(\frac{s_{n-1}+s_{n}+2\sqrt{s_{n-1}s_{n}}}{s_{n-1}+s_{n}-2\sqrt{s_{n-1}s_{n}}} \right).
    \end{equation}

Let $\Psequence$ be the space of all sequences of positive real numbers indexed by the set $\mathbb{N}_{0}$ endowed with the compact-open topology, and denote by $\mathcal{F}$ the set of all zero-twist flute surfaces, that is, $\mathcal{F}:=\{S=S(\{l_{n},0\})_{n\in\mathbb{N}_{0}}\}$. Accordingly to Theorem \ref{Teo:Parametrization-ZTFS}, there is a bijective map $\sigma: \Psequence \rightarrow \mathcal{F}$, which associates to each sequence in $\Psequence$ a zero-twist flute surface. We denote by $S_{\mathbf{x}}$ the image of $\mathbf{x}\in \Psequence$ under $\sigma$.\\

Recently, A. Basmajian, H. Hakobyan and D. {\v{S}}ari{\'c} in \cite{BasHakSa}*{Theorems 1.1 and 1.2} give sufficient conditions on Fenchel-Nielsen parameters of a surface to determine parabolicity. In the particular case of zero-twist flute surfaces, they were able to characterize parabolicty.

\begin{theorem}[\cite{BasHakSa}*{Theorem 1.5}]\label{Teo:zero_twist_tight_flute}
	A zero-twist tight flute surface $S(\{l_{n}, 0\})_{n\in\mathbb{N}}$ is of parabolic type if and only if one of the following holds:
	\begin{enumerate}
		\item The surface $S(\{l_{n}, 0\})_{n\in\mathbb{N}}$ is complete.
		\item The series $\sum\limits_{n=1}^{\infty}e^{-\frac{l_{n}}{2}}=\infty$.
	\end{enumerate}
\end{theorem}

Using equation (\ref{Eq:LengthParameterZTFS}) and Theorem \ref{Teo:Parametrization-ZTFS}, we have that items (1) and (2) in Theorem \ref{Teo:zero_twist_tight_flute} above are equivalent. Thus we obtain the following: 

\begin{corollary}\label{Cor:ParabolicZTFS}
    A zero-twist flute surface $S_{\mathbf{x}}$ is of parabolic type if and only if $\sum x_n$ diverges.
\end{corollary}

Theorem \ref{Teo:Parametrization-ZTFS} also gives an explicit set of generators for $\Gamma_{\mathbf{x}}$. We observe that if $\mathbf{x} \in \{1,2\}^{\mathbb{N}_0}$ then $\Gamma_{\mathbf{x}}$ is a subgroup of ${\rm PSL}(2,\mathbb{Z})$. This proves

\begin{corollary}
	The set of all zero-twist tight flute surfaces uniformized by subgroups of ${\rm PSL}(2,\mathbb{Z})$ can be identified to a homeomorphic copy of the Cantor set. 
\end{corollary}

\medskip

\noindent {\bf An uncountable family of hyperbolic Loch Ness monsters}. Recall that the
Loch Ness monster surface is the unique surface (up to homeomorphism) which has infinite genus and only one end (\S \ref{Subsec:TopologicalSurfaces}). In \cite{ALCA} the authors introduced a Fuchsian group uniformizing the Loch Ness monster . We generalize their construction and find an uncountable family of hyperbolic surfaces such that each one of these surface is homeomorphic to the Loch Ness Monster.

Denote by $\mathcal{N}$ the set of all sequences $\mathbf{y}:=(y_n)_{n\in \mathbb{Z}}$ of elements $y_n=(a_{n},b_{n},c_{n},d_{n},e_{n}) \in \mathbb{R}^5$ satisfying 

\begin{equation}\label{Eq:Condition}
    a_{n}<b_{n}<c_{n}<d_{n}<e_{n} \mbox{ and } e_{n}\leq a_{n+1}.
\end{equation}

Let $\mathbf{y}:=(y_n)_{n\in \mathbb{Z}} \in \mathcal{N}$. For each $n\in \mathbb{Z}$, let $f_{n}$ and $g_{n}$ be elements of ${\rm PSL}(2,\mathbb{R})$ mapping $\sigma_{n}$ onto $\tilde{\sigma}_{n}$ and $\rho_{n}$ onto $\tilde{\rho}_{n}$, respectively, where $\sigma_{n}$, $\rho_{n}$, $\tilde{\sigma}_{n}$ and $\tilde{\rho}_{n}$ are the half-circles in the hyperbolic plane $\mathbb{H}^2$ depicted in Figure \ref{Fig:half_circleIntro}. 

\begin{figure}[h!]
		\begin{center}	
		\begin{tikzpicture}[baseline=(current bounding box.north)]
		\begin{scope}[scale=0.8]
			\clip (-3.3,-0.5) rectangle (5.3,3.5);
			\draw [blue, line width=1pt] (-1,0) arc(0:180:1);
			\draw [red, line width=1pt] (0,0) arc(0:180:0.5);
			\draw [blue, line width=1pt] (3,0) arc(0:180:1.5);
			\draw [red, line width=1pt] (5,0) arc(0:180:1);
			\draw[dashed, color=red!60, thick, <-] (3.6,1) arc (0:180:2);
			\draw[dashed, color=blue!60, thick, <-] (0.6,1.3) arc (0:180:1.3);
			\draw [dashed, line width=1pt, black!30](-3,0) -- (-3,3);
			\draw [dashed, line width=1pt, black!30](5,0) -- (5,3);
			\node at (-3,-0.3) {\tiny{$a_{n}$}};
			\node at (-1,-0.3) {\tiny{$b_{n}$}};
			\node at (0,-0.3) {\tiny{$c_{n}$}};
			\node at (3,-0.3) {\tiny{$d_{n}$}};
			\node at (5,-0.3) {\tiny{$e_{n}$}};
			\node at (-2.2,1.2) {\small{{\color{blue}$\sigma_{n}$}}};
			\node at (-0.5,0.7) {\small{{\color{red}$\rho_{n}$}}};
			\node at (1.5,1.8) {\small{{\color{blue}$\tilde{\sigma}_{n}$}}};
			\node at (4,1.3) {\small{{\color{red}$\tilde{\rho}_{n}$}}};
			\node at (-0.7,2.9) {\small{$f_{n}$}};
			\node at (1.5,3.2) {\small{$g_{n}$}};
			\draw [->, >=latex, black!30](-3.3,0) -- (5.3,0);
			\end{scope}
			\end{tikzpicture} 
			\caption{\emph{Half-circles $\sigma_{n}$, $\rho_{n}$, $\tilde{\sigma}_{n}$  and $\tilde{\rho}_{n}$.}}
	\label{Fig:half_circleIntro}
	\end{center}	
\end{figure}
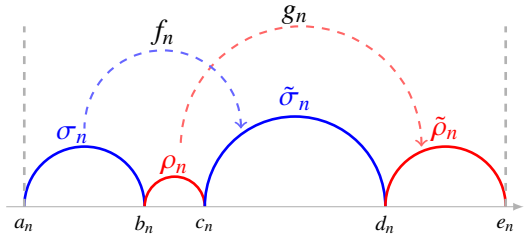

Define 
\begin{equation}
    G_\mathbf{y}:= \langle f_n,g_n:\, n\in \mathbb{Z} \rangle \leq {\rm PSL}(2,\mathbb{R}).
\end{equation}

The result we establish is the following:

\begin{theorem}\label{t:LNM}
	For each $\mathbf{y}\in \mathcal{N}$, the group $G_{\mathbf{y}}$ is Fuchsian and such that $S_\mathbf{y}:=\mathbb{H}^{2}/G_{\mathbf{y}}$ is topologically equivalent to the Loch Ness Monster. Moreover, $G_\mathbf{y}$ is of the first kind if and only if $e_n=a_{n+1}$ for all $n\in \mathbb{Z}$, $\lim\limits_{n\to \infty} e_n=\infty$ and $\lim\limits_{n\to -\infty} a_n=-\infty$.
\end{theorem}

\noindent {\bf Organization of the paper}. In Section \ref{sec:preliminaries}  we compile some classic results and concepts of the theory of surfaces necessary for the development of this work. Section \ref{sec:proof-ZTFS} is dedicated to the proof of Theorem \ref{Teo:Parametrization-ZTFS}. In Section \ref{sec:proof-LNM} we give the proof of Theorem \ref{t:LNM}.


\section{Preliminaries} \label{sec:preliminaries}

\subsection{Topological surfaces}\label{Subsec:TopologicalSurfaces}

Topological orientable surfaces are classified, up to homeomorphisms, by their genus $g(S)\in \mathbb{N}_{0}\cup\{\infty\}$, and a pair of nested topological spaces ${\rm Ends}_{\infty}(S)\subseteq {\rm Ends}(S)$ homeomorphic to a pair of nested closed subsets of the Cantor set. The spaces ${\rm Ends}(S)$ and ${\rm Ends}_{\infty}(S)$ are called the \emph{end space} and the \emph{non-planar ends space} of $S$ (or, ends accumulated by genus), respectively. Moreover, any pair of nested closed subsets of the Cantor set can be realized as the space of ends of a connected, orientable topological surface. We call to the isolated elements of the end space \emph{punctures} of the surface. For more details, we refer the reader to \cite{Ian}.

\begin{theorem}[Classification of topological surfaces, \cite{Ker}*{\S 7}, \cite{Ian}*{Theorem 1}]\label{Thm:ClassificationOfSurfaces}
	Two orientable surfaces $S_1$ and $S_2$  having the same genus are topological equivalent if and only if there exists a homeomorphism $f: {\rm Ends}(S_1)\to {\rm Ends}(S_2)$ such that $f( {\rm Ends}_{\infty}(S_1))= {\rm Ends}_{\infty}(S_2)$.
\end{theorem}

A topological surface is of \emph{infinite-type} if it has fundamental group infinitely generated. The simplest infinite-type surface that we can find are the flutes and the Loche Ness monster. 

\begin{definition}[\cite{Bas93}*{p. 423}]
The \textbf{flute surface} is the unique topological surface, up to homeomorphism, of genus zero and with ends space homeomorphic to the ordinal number $\omega +1$, see Figure \ref{Fig:topo_tight_flute}. 
    \begin{figure}[h!]
    	\begin{center}
    		\begin{tikzpicture}[baseline=(current bounding box.north)]
    		\begin{scope}[scale=0.8]
    		\clip (-0.5,-0.4) rectangle (11,3.4);
    		\draw [line width=1pt] (0.5,0) -- (9.8,0);
    		\draw [line width=1pt] (9.5,1) -- (9.8,1.01);
    		\draw [dashed, line width=1pt] (0.5,0.3) ellipse (1mm and 3mm);
    		%
    		\draw [line width=1pt] (0.5,0.6) to[out=0,in=-90] (1.5,2.5);
    		\draw [dashed, line width=1pt] (2,2.5) ellipse (5mm and 2mm);
    		\node at (2,3.1) {$\vdots$};
    		\draw [line width=1pt] (2.5,2.5) to[out=-90,in=180] (3.5,1);
    		\draw [line width=1pt] (3.5,1) to[out=0,in=-90] (4.5,2.5);
    		\draw [dashed, line width=1pt] (5,2.5) ellipse (5mm and 2mm);
    		\node at (5,3.1) {$\vdots$};
    		\draw [line width=1pt] (5.5,2.5) to[out=-90,in=180] (6.5,1);
    		\draw [line width=1pt] (6.5,1) to[out=0,in=-90] (7.5,2.5);
    		%
    		\draw [dashed, line width=1pt] (8,2.5) ellipse (5mm and 2mm);
    		\node at (8,3.1) {$\vdots$};
    		\draw [line width=1pt] (8.5,2.5) to[out=-90,in=180] (9.5,1);
    		\node at (10.3,0.5) {$\ldots$};
    		\draw [dashed, line width=0.6pt] (3.5,0.5) ellipse (2mm and 5mm);
    		\draw [line width=1pt] (3.5,1) arc
    		[
    		start angle=90,
    		end angle=270,
    		x radius=2mm,
    		y radius =5mm
    		] ;
    		\draw [dashed, line width=0.6pt] (6.5,0.5) ellipse (2mm and 5mm);
    		\draw [line width=1pt] (6.5,1) arc
    		[
    		start angle=90,
    		end angle=270,
    		x radius=2mm,
    		y radius =5mm
    		] ;
    		\draw [dashed, line width=0.6pt] (9.5,0.5) ellipse (2mm and 5mm);
    		\draw [line width=1pt] (9.5,1) arc
    		[
    		start angle=90,
    		end angle=270,
    		x radius=2mm,
    		y radius =5mm
    		] ;
    		\end{scope}
    		\end{tikzpicture}
    	\end{center}
    	\caption{\emph{A flute surface.}}
    	\label{Fig:topo_tight_flute}
    \end{figure}
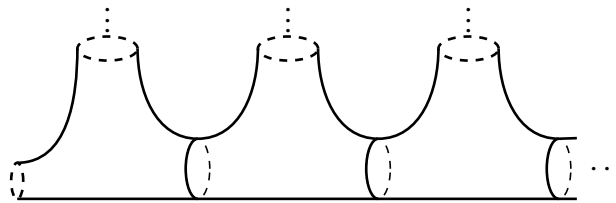
\end{definition}

\begin{definition}
	The \textbf{Loch Ness monster} is the unique infinite-type surface, up to homeomorphism, of infinite genus with exactly one end, see Figure \ref{Fig:LNM}.
\begin{figure}[h!]
	\centering
	\begin{tikzpicture}[baseline=(current bounding box.north)]  
	\begin{scope}[scale=0.45]
	\clip (-5.8,-5) rectangle (12,3.9);
	\draw[line width=1pt] (-3,0) arc (45:315:1.5 and 1);
	\draw [line width=1pt] (-4.8,-0.55) arc
	[
	start angle=180,
	end angle=360,
	x radius=6mm,
	y radius =3mm
	] ;
	\draw [line width=1pt] (-3.82,-0.8) arc
	[
	start angle=-20,
	end angle=200,
	x radius=4mm,
	y radius =2mm
	] ;
	\draw [dashed, line width=0.8pt] (-3,-1.4) arc
	[
	start angle=-90,
	end angle=90,
	x radius=3mm,
	y radius =7mm
	] ;
	\draw [dashed, line width=0.8pt] (-3,0) arc
	[
	start angle=90,
	end angle=270,
	x radius=3mm,
	y radius =7mm
	] ;
	\draw [line width=1pt] (-3,0) to[out=0,in=-90] (-2,1);
	\draw [dashed, line width=0.8pt] (-0.6,1) arc
	[
	start angle=0,
	end angle=180,
	x radius=7mm,
	y radius =3mm
	] ;
	\draw [dashed, line width=0.8pt] (-2,1) arc
	[
	start angle=180,
	end angle=360,
	x radius=7mm,
	y radius =3mm
	] ;
	\draw[line width=1pt] (-0.6,1) arc (-45:225:1 and 1.5);
	\draw [line width=1pt] (-1.15,2.7) arc
	[
	start angle=90,
	end angle=270,
	x radius=3mm,
	y radius =6mm
	] ;
	\draw [line width=1pt] (-1.35,2.5) arc
	[
	start angle=100,
	end angle=-100,
	x radius=2mm,
	y radius =4mm
	] ;
	\draw [line width=1pt] (-0.6,1) to[out=-90,in=180] (0.4,0);
	\draw [dashed, line width=0.8pt] (0.4,-1.4) arc
	[
	start angle=-90,
	end angle=90,
	x radius=3mm,
	y radius =7mm
	] ;
	\draw [dashed, line width=0.8pt] (0.4,0) arc
	[
	start angle=90,
	end angle=270,
	x radius=3mm,
	y radius =7mm
	] ;
	\draw [line width=1pt] (0.4,-1.4) to[out=0,in=90] (1.4,-2.4);
	\draw [dashed, line width=0.8pt] (2.8,-2.4) arc
	[
	start angle=0,
	end angle=180,
	x radius=7mm,
	y radius =3mm
	] ;
	\draw [dashed, line width=0.8pt] (1.4,-2.4) arc
	[
	start angle=180,
	end angle=360,
	x radius=7mm,
	y radius =3mm
	] ;
	\draw [line width=1pt] (2.8,-2.4) to[out=90,in=180] (3.8,-1.4);
	\draw [dashed, line width=0.8pt] (3.8,0) arc
	[
	start angle=90,
	end angle=270,
	x radius=3mm,
	y radius =7mm
	] ;
	\draw [dashed, line width=0.8pt] (3.8,0) arc
	[
	start angle=90,
	end angle=-90,
	x radius=3mm,
	y radius =7mm
	] ;	
	\draw [line width=1pt] (2.8,-2.4) arc (45:-225:1 and 1.5);
	\draw [line width=1pt] (2.2,-3) arc
	[
	start angle=90,
	end angle=270,
	x radius=3mm,
	y radius =6mm
	] ;
	\draw [line width=1pt] (2,-3.2) arc
	[
	start angle=100,
	end angle=-100,
	x radius=2mm,
	y radius =4mm
	] ;
	\draw [line width=1pt](-3,-1.4) -- (0.4,-1.4);
	\draw [line width=1pt](0.4,0) -- (3.8,0);
	\node at (11.5,-0.7) {$\ldots$};
	\draw [line width=1pt] (3.8,0) to[out=0,in=-90] (4.8,1);
	\draw[line width=1pt] (6.2,1) arc (-45:225:1 and 1.5);
	\draw [line width=1pt] (5.6,2.7) arc
	[
	start angle=90,
	end angle=270,
	x radius=3mm,
	y radius =6mm
	] ;
	\draw [line width=1pt] (5.4,2.5) arc
	[
	start angle=100,
	end angle=-100,
	x radius=2mm,
	y radius =4mm
	] ;
	\draw [dashed, line width=0.8pt] (6.2,1) arc
	[
	start angle=0,
	end angle=180,
	x radius=7mm,
	y radius =3mm
	] ;	
	\draw [dashed, line width=0.8pt] (6.2,1) arc
	[
	start angle=0,
	end angle=-180,
	x radius=7mm,
	y radius =3mm
	] ;	
	\draw [line width=1pt] (6.2,1) to[out=-90,in=180] (7.2,0);
	\draw [line width=1pt](3.8,-1.4) -- (7.2,-1.4);
	\draw [dashed, line width=0.8pt] (7.2,0) arc
	[
	start angle=90,
	end angle=270,
	x radius=3mm,
	y radius =7mm
	] ;
	\draw [dashed, line width=0.8pt] (7.2,0) arc
	[
	start angle=90,
	end angle=-90,
	x radius=3mm,
	y radius =7mm
	] ;	
	\draw [line width=1pt] (7.2,-1.4) to[out=0,in=90] (8.2,-2.4);
	\draw [line width=1pt] (9.6,-2.4) arc (45:-225:1 and 1.5);
	\draw [line width=1pt] (9,-3) arc
	[
	start angle=90,
	end angle=270,
	x radius=3mm,
	y radius =6mm
	] ;
	\draw [line width=1pt] (8.8,-3.2) arc
	[
	start angle=100,
	end angle=-100,
	x radius=2mm,
	y radius =4mm
	] ;
	\draw [dashed, line width=0.8pt] (9.6,-2.4) arc
	[
	start angle=0,
	end angle=180,
	x radius=7mm,
	y radius =3mm
	] ;	
	\draw [dashed, line width=0.8pt] (9.6,-2.4) arc
	[
	start angle=0,
	end angle=-180,
	x radius=7mm,
	y radius =3mm
	] ;	
	\draw [line width=1pt] (9.6,-2.4) to[out=90,in=180] (10.6,-1.4);
	\draw [line width=1pt](7.2,0) -- (10.6,0);
	\draw [dashed, line width=0.8pt] (10.6,0) arc
	[
	start angle=90,
	end angle=270,
	x radius=3mm,
	y radius =7mm
	] ;	
	\draw [dashed,  line width=0.8pt] (10.6,0) arc
	[
	start angle=90,
	end angle=-270,
	x radius=3mm,
	y radius =7mm
	] ;	
	\end{scope}
	\end{tikzpicture}
	\caption{\emph{The Loch Ness monster.}}
	\label{Fig:LNM}
\end{figure}
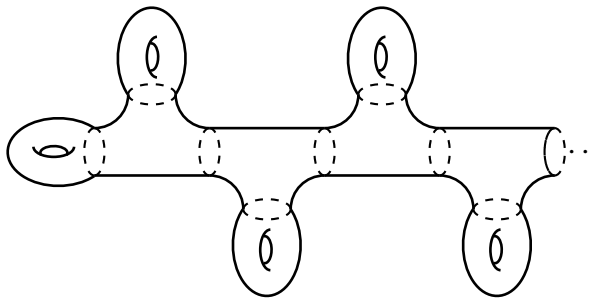
\end{definition}

In section \ref{section:Loch_ness_monster} we use the following result to prove the first part of Theorem \ref{t:LNM}.

\begin{lemma}[\cite{SPE}*{\S 5.1., p. 320}]\label{lemma:spec} 
	The surface $S$ has exactly one end if and only if for all compact subset $K \subset S$ there is a compact subset $K^{'}\subset S$ such that $K\subset K^{'}$ and $S \setminus  K^{'}$ is connected.
\end{lemma}

We say that a simple close curved on a surface $S$ is \emph{essential} if it is not isotopic to the boundary  loop of a disk or a punctured disk. Recall that a \emph{topological pair of pants} is a topological surface homeomorphic to a three punctured sphere. 

\begin{definition}\label{Def:PairPantsDecomposition}
    We say that a collection of pairwise disjoint essential curves in a surface $S$ is a \textbf{pair of pants decomposition} of $S$ if it decomposes the surface into a disjoint union of pair of pants.    
\end{definition}

Observe that any topological surface admits a pair of pants decomposition \cite{Alvarez2004}.


\subsection{Fenchel-Nielsen parameters of a hyperbolic surfaces}\label{Subsec:FenchelNielsen} For the results presented in this section we refer the reader to \cite{FaMa2012}*{Chapter 10} for the definition of the Fenchel-Nielsen parameters in the context of finite-type surfaces, and to \cite{AlLiPaWeSun2011} for surfaces of infinite-type.

A \emph{geodesic pair of pants} is a hyperbolic surface $P$ of finite hyperbolic area such that its interior is homeomorphic to a pair of pants and with at least one boundary a closed geodesic, see Figure \ref{Fig:GeodesicPairPants}. A \emph{tight pair of pants} is a pair of pants that has exactly one puncture, see Figure \ref{Fig:GeodesicPairPants}-b. Similarly, a collection of pairwise disjoint of essential geodesic curves in a surface $S$ is \emph{geometric pair of pants decomposition} of $S$ if it decomposes the surface into geodesic pair of pants. In \cite{AlLiPaWeSun2011}*{Theorem 4.5} the authors give sufficient conditions under which a topological and geodesic decomposition of a surface in pair of pants is related .    
\begin{figure}[h!]
	\centering
	\begin{tabular}{ccc}
		\begin{tikzpicture}[baseline=(current bounding box.north)]
		\begin{scope}[scale=0.8]
		\clip (-4.4,-0.9) rectangle (0.4,2.6);
		\draw [dashed, line width=1pt] (-4,0) ellipse (2mm and 5mm);
		\draw [dashed, line width=1pt] (-2,2.3) ellipse (5mm and 2mm);
		\draw [dashed, line width=1pt]  (0,0) ellipse (2mm and 5mm);
		\draw [line width=1pt](-4,-0.5) to[out=0,in=-180] (0,-0.5);
		\draw [line width=1pt] (-4,0.5) to[out=0,in=-90] (-2.5,2.3);
		\draw [line width=1pt] (-1.5,2.3) to[out=-90,in=180] (0,0.5);
		\end{scope}
		\end{tikzpicture} &  \begin{tikzpicture}[baseline=(current bounding box.north)]
		\begin{scope}[scale=0.8]
		\clip (-0.4,-0.9) rectangle (4.4,2.6);
		\draw [dashed, line width=1pt]  (0,0) ellipse (2mm and 5mm);
		\draw [dashed, line width=1pt] (4,0) ellipse (2mm and 5mm);
		\draw [line width=1pt](0,-0.5) to[out=0,in=-180] (4,-0.5);
		\draw [line width=1pt] (0,0.5) to[out=0,in=-90] (1.98,2.5);
		\draw [line width=1pt] (2.02,2.5) to[out=-90,in=180] (4,0.5);
		\end{scope}
		\end{tikzpicture} &  \begin{tikzpicture}[baseline=(current bounding box.north)]
		\begin{scope}[scale=0.8]
		\clip (-0.4,-0.9) rectangle (4.4,2.6);
		\draw [dashed, line width=1pt]  (0,1) ellipse (2mm and 5mm);
		\draw [line width=1pt] (0,1.5) to[out=0,in=-180] (2.5,2.5);
		\draw [line width=1pt] (0,0.5) to[out=0,in=180] (2.5,-0.5);
		\draw [line width=1pt] (2.5,2.5) to[out=-160,in=160] (2.5,-0.5);
		\end{scope}
		\end{tikzpicture}  \\
		a)  & b)  & c) \\
	\end{tabular}
	\caption{\textit{Three possible kinds of hyperbolic pants with geodesic boundary. In particular, b) represents a tight pair of pants.}} 
	\label{Fig:GeodesicPairPants}
\end{figure}
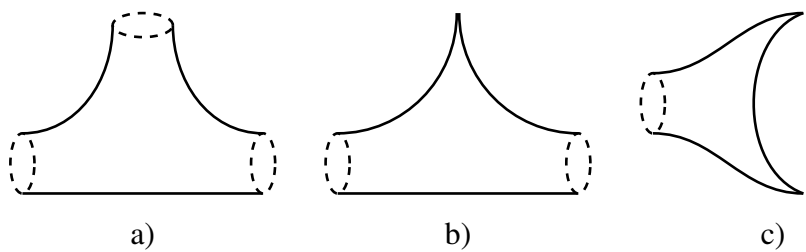

\subsubsection{Length and twist parameters} Let $P$ be a geodesic pair of pants and fix $\alpha$ a boundary curve of $P$. We choose a \emph{marked point} $x$ in $\alpha$ in the following way: Let $\beta$ be either a boundary component (different from $\alpha$)  or a puncture of $P$. Let $\gamma$ be the orthogeodesics between $\alpha$ and $\beta$. Then the marked point $x$ in $\alpha$ is defined as the intersection point of $\alpha$ with $\gamma$. In Figure \ref{Fig:Orthogeodesic} is shown such marked point $x$ in $\alpha$ which is obtained from drawing the orthogeodesic $\gamma$ connecting the boundary component $\alpha$ with $\beta$. The hyperbolic length  of the geodesic $\alpha$ is denoted by $l(\alpha)$. 

\begin{figure}[h!]
	\centering
	\begin{tabular}{cc}
		\begin{tikzpicture}[baseline=(current bounding box.north)]
		\begin{scope}[scale=0.8]
		\clip (-4.4,-1.1) rectangle (0.4,2.6);
		\draw [dashed, red, line width=1pt] (-4,0) ellipse (2mm and 6mm);
		\draw [dashed, line width=1pt] (-2,2.3) ellipse (6mm and 2mm);
		\draw [dashed, line width=1pt]  (0,0) ellipse (2mm and 6mm);
		\draw [line width=1pt](-4,-0.6) to[out=0,in=-180] (0,-0.6);
		\draw [blue, line width=1pt] (-4,0.6) to[out=0,in=-90] (-2.6,2.3);
		\draw [line width=1pt] (-1.4,2.3) to[out=-90,in=180] (0,0.6);
		\draw [blue, line width=1pt] (-3.5,0.5)--(-3.6,0.66);
		\draw [blue, line width=1pt] (-3.85,0.43)--(-3.5,0.5);
		\node at (-4,0.9) {{\small $x$}};
		\node at (-4,-0.9) {{\small {\color{red}$\alpha$}}};
		\node at (-1,2.3) {{\small $\beta$}};
		\node at (-2.6,1) {{\small {\color{blue}$\gamma$}}};
		\node at (-4,0.57) {{\small $\bullet$}};
		\node at (-2,0.1) {{\small $P$}};
		\end{scope}
		\end{tikzpicture} &  \begin{tikzpicture}[baseline=(current bounding box.north)]
		\begin{scope}[scale=0.8]
		\clip (-0.4,-1.1) rectangle (4.4,2.6);
		\draw [dashed, red, line width=1pt]  (0,0) ellipse (2mm and 6mm);
		\draw [dashed, line width=1pt] (4,0) ellipse (2mm and 6mm);
		\draw [line width=1pt](0,-0.6) to[out=0,in=-180] (4,-0.6);
		\draw [blue, line width=1pt] (0,0.6) to[out=0,in=-90] (1.98,2.5);
		\draw [line width=1pt] (2.02,2.5) to[out=-90,in=180] (4,0.6);
		\draw [blue, line width=1pt] (0.5,0.46)--(0.4,0.66);
		\draw [blue, line width=1pt] (0.15,0.4)--(0.5,0.46);
		\node at (0,0.9) {\small{$x$}};
		\node at (0,-0.9) {\small{ {\color{red}$\alpha$}}};
		\node at (2.5,2.3) {\small{$\beta$}};
		\node at (1.8,1) {{\small {\color{blue}$\gamma$}}};
		\node at (0,0.57) {{\small $\bullet$}};
		\node at (2,0.1) {{\small $P$}};
		\end{scope}
		\end{tikzpicture}\\
	\end{tabular}
	\caption{\emph{The point $x \in \alpha$ is the intersection of $\alpha$ with the orthogeodesic $\gamma$.}}
	\label{Fig:Orthogeodesic}
\end{figure}
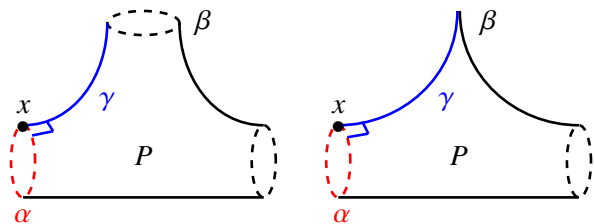

 Let $P$ and $P^\prime$ be two geodesic pair of pants with $\alpha$ and $\alpha'$ boundary geodesic curve of $P$ and $P^\prime$, respectively.  Let $\gamma$ and $\gamma'$ be the geodesic in  $P$ and $P'$, respectively, that define the marked points $x$ in $\alpha$ and $x'$ in $\alpha'$ as above. Suppose that $\alpha$ and $\alpha'$ have the same hyperbolic length, that is, $l(\alpha)=l(\alpha')$. Then we identify  $\alpha$ and $\alpha^\prime$, by gluing with an isometry $T$, to obtain a surface $S$ from  $P$ and $P'$. This isometric identification is determined by the pair $\{l(\alpha),t(\alpha)\}$, such that the \emph{length parameter} $l(\alpha)$ is  the hyperbolic length of $\alpha$, and the \emph{twist parameter} $t(\alpha) \in \left(-\frac{1}{2},\frac{1}{2}\right]$ corresponds to the relative position of the marked points $x$ in $\alpha$ and $x'$ in $\alpha'$, which is given as follows: 
 \begin{itemize}
 \item[$\ast$] If $x=T(x')$, then $t(\alpha)=0$.
 
 \item[$\ast$] If $x\neq T(x')$, then $\vert t(\alpha)\vert \leq \frac{1}{2}$ is equal to the hyperbolic length of the shorter arc $\sigma$ having endpoints $x$ and $T(x')$ contained in $\alpha \setminus \{x,T(x')\}$, divided by $l(\alpha)$. If $\vert t(\alpha)\vert = \frac{1}{2}$, then $t(\alpha)=\frac{1}{2}$. If $\vert t(\alpha) \vert < \frac{1}{2}$ then we orient $\alpha$ with the orientation induced from $P$. If $\sigma$ is the arc from $x$ to $T(x^\prime)$ then $t(\alpha)=\vert t(\alpha) \vert$; otherwise $t(\alpha)=-\vert t(\alpha) \vert$. See Figure \ref{Fig:TwistParameter}. a). 
 \end{itemize}

 \begin{figure}[h!]
 	\centering
 	\begin{tabular}{ccc}
 		\begin{tikzpicture}[baseline=(current bounding box.north)]
 		\begin{scope}[scale=1]
 		\clip (-2.4,-0.7) rectangle (3.4,3.4);
 		\draw [dashed, line width=1pt] (-2,-0.5) arc
 		[
 		start angle=-90,
 		end angle=90,
 		x radius=2mm,
 		y radius =5mm
 		] ;
 		\draw [dashed, line width=1pt] (-2,0.5) arc
 		[
 		start angle=90,
 		end angle=270,
 		x radius=2mm,
 		y radius =5mm
 		] ;
 		\draw [dashed, line width=1pt] (-2,1.5) arc
 		[
 		start angle=-90,
 		end angle=90,
 		x radius=2mm,
 		y radius =5mm
 		] ;
 		\draw [dashed, line width=1pt] (-2,2.5) arc
 		[
 		start angle=90,
 		end angle=270,
 		x radius=2mm,
 		y radius =5mm
 		] ;
 		\draw [line width=1pt]  (-0.5,1) ellipse (2mm and 15mm);
 		\draw [<-<, >=latex, red, line width=1pt] (-0.373,2.15) arc
 		[
 		start angle=50,
 		end angle=147,
 		x radius=2mm,
 		y radius =15mm
 		] ;
 		\draw [line width=1pt](-2,-0.5) -- (-0.5,-0.5);
 		\draw [line width=1pt] (-2,0.5) to[out=0,in=0] (-2,1.5);
 		\draw [line width=1pt] (-2,2.5) -- (-0.5,2.5);
 		\draw [blue, line width=1pt] (-2.2,1.84) -- (-0.66,1.84);
 		\draw [blue, line width=0.7pt] (-0.9,1.7) -- (-0.67,1.7);
 		\draw [blue, line width=0.7pt] (-0.9,1.7) -- (-0.9,1.84);
 		\node at (-0.66,1.8) {{\small$\bullet$}};
 		\node at (-0.9,2.1) {{\small $x$}};
 		\node at (-0.37,2.15) {{\small $\bullet$}};
 		\node at (1.4,2.15) {{\small $\bullet$}};
 		\node at (1.2,2.5) {{\small $x'$}};
 		\node at (0.2,2.5) {{\small $T(x')$}};
 		\node at (-0.6,2.8) {{\color{red}$\sigma$}};
 		\node at (-1.3,1.5) {{\color{blue}$\gamma$}};
 		\node at (2.3,1.8) {{\color{blue}$\gamma'$}};
 		\node at (-1,1) {{\small $\alpha$}};
 		\node at (-1.3,0) {{\small $P$}};
 		\node at (2,1) {{\small $\alpha'$}};
 		\node at (2.3,0) {{\small $P'$}};
 		\draw [dashed, blue, line width=1pt] (1.4,2.15) -- (2.8,2.15);
 		\draw [blue, line width=0.7pt] (1.35,2) -- (1.55,2);
 		\draw [blue, line width=0.7pt] (1.55,2) -- (1.55,2.15);
 		\draw [<-, >=latex, line width=1pt] (-0.2,2.15) -- (1.2,2.15);
 		\draw [line width=1pt]  (1.5,1) ellipse (2mm and 15mm);
 		\draw [dashed, line width=1pt] (3,-0.5) arc
 		[
 		start angle=-90,
 		end angle=90,
 		x radius=2mm,
 		y radius =5mm
 		] ;
 		\draw [dashed, line width=1pt] (3,0.5) arc
 		[
 		start angle=90,
 		end angle=270,
 		x radius=2mm,
 		y radius =5mm
 		] ;
 		\draw [dashed, line width=1pt] (3,1.5) arc
 		[
 		start angle=-90,
 		end angle=90,
 		x radius=2mm,
 		y radius =5mm
 		] ;
 		\draw [dashed, line width=1pt] (3,2.5) arc
 		[
 		start angle=90,
 		end angle=270,
 		x radius=2mm,
 		y radius =5mm
 		] ;
 		\draw [line width=1pt](3,-0.5) -- (1.5,-0.5);
 		\draw [line width=1pt] (3,0.5) to[out=180,in=180] (3,1.5);
 		\draw [line width=1pt] (3,2.5) -- (1.5,2.5);
 		\end{scope}
 		\end{tikzpicture}& & \begin{tikzpicture}[baseline=(current bounding box.north)]
 		\begin{scope}[scale=1]
 		\clip (-2.4,-0.7) rectangle (1.4,3.4);
 		\draw [dashed, line width=1pt] (-2,-0.5) arc
 		[
 		start angle=-90,
 		end angle=90,
 		x radius=2mm,
 		y radius =5mm
 		] ;
 		\draw [dashed, line width=1pt] (-2,0.5) arc
 		[
 		start angle=90,
 		end angle=270,
 		x radius=2mm,
 		y radius =5mm
 		] ;
 		\draw [dashed, line width=1pt] (-2,1.5) arc
 		[
 		start angle=-90,
 		end angle=90,
 		x radius=2mm,
 		y radius =5mm
 		] ;
 		\draw [dashed, line width=1pt] (-2,2.5) arc
 		[
 		start angle=90,
 		end angle=270,
 		x radius=2mm,
 		y radius =5mm
 		] ;
 		\draw [red, dashed, line width=0.3pt] (-0.5,-0.5) arc
 		[
 		start angle=-90,
 		end angle=90,
 		x radius=2mm,
 		y radius =15mm
 		] ;
 		\draw [red, line width=1pt] (-0.5,2.5) arc
 		[
 		start angle=90,
 		end angle=270,
 		x radius=2mm,
 		y radius =15mm
 		] ;
 		\draw [line width=1pt](-2,-0.5) -- (-0.5,-0.5);
 		\draw [line width=1pt] (-2,0.5) to[out=0,in=0] (-2,1.5);
 		\draw [line width=1pt] (-2,2.5) -- (-0.5,2.5);
 		\draw [blue, line width=1pt] (-2.2,1.84) -- (-0.66,1.84);
 		\draw [blue, line width=0.7pt] (-0.9,1.7) -- (-0.67,1.7);
 		\draw [blue, line width=0.7pt] (-0.9,1.7) -- (-0.9,1.84);
 		\node at (-0.66,1.8) {{\small$\bullet$}};
 		\node at (-0.9,2.1) {{\small $x$}};
 		\node at (-0.4,2.1) {{\small $x'$}};
 		\node at (-1.3,1.5) {{\color{blue}$\gamma$}};
 		\node at (0.3,1.5) {{\color{blue}$\gamma'$}};
 		\node at (-1,1) {{\small {\color{red}$\alpha$}}};
 		\node at (-1.3,0) {{\small $P$}};
 		\node at (-0.49,1) {{\small {\color{red}$\alpha'$}}};
 		\node at (0.3,0) {{\small $P'$}};
 		\draw [blue, line width=1pt] (-0.67,1.84) -- (1.2,1.84);
 		\draw [blue, line width=0.7pt] (-0.67,1.7) -- (-0.4,1.7);
 		\draw [blue, line width=0.7pt] (-0.4,1.7) -- (-0.4,1.84);
 		\draw [dashed, line width=1pt] (1,-0.5) arc
 		[
 		start angle=-90,
 		end angle=90,
 		x radius=2mm,
 		y radius =5mm
 		] ;
 		\draw [dashed, line width=1pt] (1,0.5) arc
 		[
 		start angle=90,
 		end angle=270,
 		x radius=2mm,
 		y radius =5mm
 		] ;
 		\draw [dashed, line width=1pt] (1,1.5) arc
 		[
 		start angle=-90,
 		end angle=90,
 		x radius=2mm,
 		y radius =5mm
 		] ;
 		\draw [dashed, line width=1pt] (1,2.5) arc
 		[
 		start angle=90,
 		end angle=270,
 		x radius=2mm,
 		y radius =5mm
 		] ;
 		\draw [line width=1pt](-0.5,-0.5) -- (1,-0.5);
 		\draw [line width=1pt] (1,0.5) to[out=180,in=180] (1,1.5);
 		\draw [line width=1pt] (1,2.5) -- (-0.5,2.5);
 		\end{scope}
 		\end{tikzpicture}\\
 		\emph{a). Twist parameter.}& &\emph{b). If $t(\alpha)=0$},\\
 			                     &&\emph{then $\gamma \cup \gamma'$ is orthogonal to $\alpha$.}\\
 	\end{tabular}
 \caption{}
 \label{Fig:TwistParameter}
 \end{figure}

\begin{remark}\label{remark:BiinfiniteGeodesic}
	 If the twist parameter $t(\alpha)=0$, then $\gamma \cup \gamma'$ is a geodesic in $S$ that is orthogonal to the closed geodesic $\alpha$, see Figure \ref{Fig:TwistParameter}. b). 
\end{remark}

\begin{definition}
Let $S$ be a surface obtained by gluing countably many geodesic pair of pants in the way described above. These gluing have associated a geometric pair of pants  decomposition $\{\alpha_i: i \in \mathbb{N}\}$ of $S$. Let $l_{i}(\alpha)=l_i$ and $t_{i}(\alpha)=t_i$ be the length and twist parameter respectively of $\alpha_i$. The collection of pairs 
\begin{equation}
(\{l_{i},t_{i}\})_{i\in\mathbb{N}},
\end{equation}
is called the \textbf{Fenchel-Nielsen parameters} of $S$. As the hyperbolic metric on $S$ is uniquely determined by its Fenchel-Nielsen parameters, then 
\[
S:=S(\{l_{i},t_{i}\})_{i\in\mathbb{N}}.
\] 
\end{definition}

The surface $S$ might not be complete in the induced hyperbolic metric. V. \'Alvarez and J. M. Rodr\'iguez in \cite{Alvarez2004} showed that the boundary of the metric completion of $S$ consists of simple closed geodesics and bi-infinite simple geodesics. Moreover, they proved that by attaching funnels to the closed geodesics and attaching geodesic half-planes to the
bi-infinite geodesics of the boundary of the metric completion of $S$, we obtain a surface $\hat{S}$ homeomorphic to $S$ with a geodesically complete hyperbolic metric such that the inclusion $i: S \hookrightarrow \hat{S}$ is an isometric embedding. Conversely, any geodesically complete surface is obtained by attaching funnels and half-planes to the convex core of the surface, see also \cite{BasmajianSaric}.

\medskip

We are in a position to define the so-called \emph{tight flute surfaces}.

\begin{definition}[\cite{Bas93}*{p. 423}]\label{definition:tight_flute_surface}
	A \textbf{tight flute surface} $S$ is a surface obtained by starting with a geodesic pair of pants $P_0$ with two punctures and then consecutively gluing tight pairs of pants $P_n$, $n\geq 1$. 
\end{definition}	
	
If $S$ is a tight flute surface, we denote by $\alpha_n$ the closed geodesic of the boundary of the surface obtained after gluing $n$ geodesic pairs of pants and let $l_n$ and $t_n$ be the length and twist parameters associated to $\alpha_n$. In terms of Fenchel-Nielsen parameters this surface is denoted by $S:=S(\{l_n, t_n\})_{n\in\mathbb{N}}$. 

\begin{definition}
If all twist parameters of a tight flute surface are zero, that is $S:=S(\{l_n,0\})_{n\in\mathbb{N}}$, then we called it a \textbf{zero-twist tight flute surface}. 
\end{definition}




\section{Proof of Theorem \ref{Teo:Parametrization-ZTFS}}\label{sec:proof-ZTFS}  

We begin by associating to each  zero-twist flute surface $S$ a sequence $\textbf{x}$ of positive real numbers and a Fuchsian group $\Gamma_{\mathbf{x}}$ such that the convex core of $\mathbb{H}^2/\Gamma_{\mathbf{x}}$ is isometric to $S$. We do this by constructing explicitly a hyperbolic ideal polygon $\mathcal{P}$ by cutting $S$ along an infinite collection of bi-infinite geodesics which serves as fundamental domain for $\Gamma_{\mathbf{x}}$.

\medskip

Let $S=S(\{l_n,0\})_{n\in\mathbb{N}_{0}}$ be a zero-twist flute surface (see Figure \ref{Fig:tight_flute_1}). For each $n\in\mathbb{N}$, let $\alpha_{n}$ be the closed geodesic in $S$, which comes from gluing the geodesic pair of pants $P_{n-1}$ and $P_{n}$. Let $\textbf{0}$ and $s_{0}$ be the punctures of $P_0$ and for $n\geq 1$, let $s_{n}$ be the unique puncture of $P_n$. Now we describe how to obtain $\mathcal{P}$ by removing a collection $\{\gamma_n\}_{n\in \mathbb{N}_0}$ of geodesics in $S$ having endpoints on the punctures of $S$. Let be the geodesics $\gamma_0 \subset P_0$ connecting the two punctures of $P_0$, and $\gamma_n$ connecting the punctures  $s_{n-1}$ and $s_n$, the last one is orthogonal to the close geodesic $\alpha_{n-1}$, see Remark \ref{remark:BiinfiniteGeodesic}. Finally, we draw on $S$ the geodesic ray $\beta$ orthogonal to each $\alpha_{n}$ and having one endpoint in $\textbf{0}$ (see Figure \ref{Fig:tight_flute_1}). 
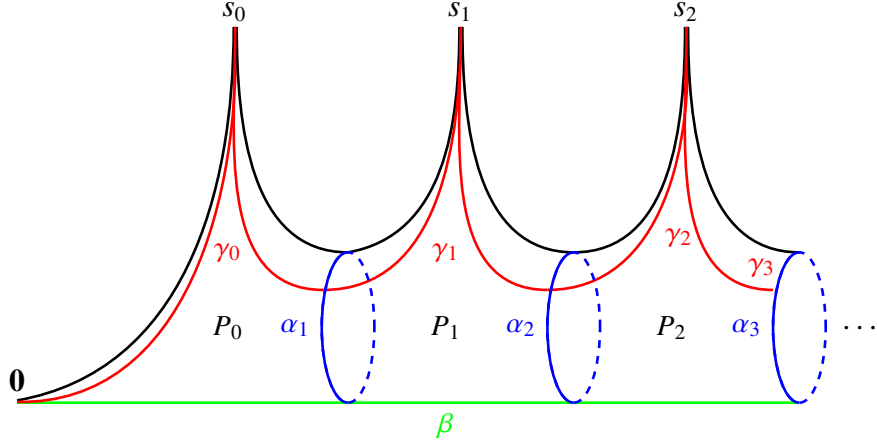
\begin{figure}[h!]
	\begin{center}
		\begin{tikzpicture}[baseline=(current bounding box.north)]
		\begin{scope}[scale=1]
		\clip (0,-0.5) rectangle (11.5,5.5);
		\draw [green, line width=1pt] (0.1,0) to[out=0,in=180] (4.5,0);
		\draw [green, line width=1pt] (4.5,0) to[out=0,in=180] (7.5,0);
		\draw [green, line width=1pt] (7.5,0) to[out=0,in=180] (10.5,0);
		\node at (0.1,0.3) {$\textbf{0}$};
		\node at (2.98,5.2) {$s_{0}$};
		\node at (5.98,5.2) {$s_{1}$};
		\node at (8.98,5.2) {$s_{2}$};
		\draw [line width=1pt] (0.1,0.03) to[out=10,in=-90] (2.98,5);
		\draw [red, line width=1pt] (0.1,0.01) to[out=-1,in=-88] (3,5);
		\node at (2.9,2) {{\small {\color{red} $\gamma_{0}$}}};
		\draw [line width=1pt] (3.02,5) to[out=-90,in=180] (4.5,2);
		\draw [line width=1pt] (4.5,2) to[out=10,in=-90] (5.98,5);
		\draw [red, line width=1pt] (3,5) to[out=-90,in=180] (4.17,1.5);
		\draw [red, line width=1pt] (4.17,1.5) to[out=0,in=-90] (6,5);
	    \node at (5.8,2) {{\small {\color{red}$\gamma_{1}$}}};
		\draw [line width=1pt] (6.02,5) to[out=-90,in=180] (7.5,2);
		\draw [line width=1pt] (7.5,2) to[out=0,in=-90] (8.98,5);
		\draw [red, line width=1pt] (6,5) to[out=-90,in=180] (7.15,1.5);
		\draw [red, line width=1pt] (7.15,1.5) to[out=0,in=-86] (9,5);
		\node at (8.9,2.2) {{\small{\color{red}$\gamma_{2}$}}};
		\draw [line width=1pt] (9.02,5) to[out=-90,in=180] (10.5,2);
		\draw [red, line width=1pt] (9,5) to[out=-90,in=180] (10.15,1.5);
		\node at (10,1.8) {{\small {\color{red}$\gamma_{3}$}}};
		\draw [blue, dashed, line width=1pt] (4.5,1) ellipse (3.5mm and 10mm);
		\draw [blue, line width=1pt] (4.5,2) arc
		[
		start angle=90,
		end angle=270,
		x radius=3.5mm,
		y radius =10mm
		] ;
		\node at (3.8,1) {{\small {\color{blue}$\alpha_{1}$}}};
		\draw [blue, dashed, line width=1pt] (7.5,1) ellipse (3.5mm and 10mm);	
		\draw [blue, line width=1pt] (7.5,2) arc
		[
		start angle=90,
		end angle=270,
		x radius=3.5mm,
		y radius =10mm
		] ;
		\node at (6.8,1) {{\small {\color{blue}$\alpha_{2}$}}};
		\draw [blue, dashed, line width=1pt] (10.5,1) ellipse (3.5mm and 10mm);	
		\draw [blue, line width=1pt] (10.5,2) arc
		[
		start angle=90,
		end angle=270,
		x radius=3.5mm,
		y radius =10mm
		] ;
		\node at (9.8,1) {{\small {\color{blue}$\alpha_{3}$}}};
		\node at (11.3,1) {$\ldots$};
		\node at (2.9,1) {{\small$P_{0}$}};
		\node at (5.8,1) {{\small$P_{1}$}};
		\node at (5.8,-0.3) {{\small {\color{green}$\beta$}}};
		\node at (8.8,1) {{\small$P_{2}$}};
		\end{scope}
		\end{tikzpicture}
	\end{center}
	\caption{\emph{A zero-twist tight flute surface.}}
	\label{Fig:tight_flute_1}
\end{figure}
Note that $S$ admits a reflection $\tau$ whose connected components of fixed points are the geodesic arcs $ \gamma_ {n} $ and $ \beta $.

Now, we cut $S$ along each $\gamma_{n}$. Then $S$ turns into an ideal hyperbolic polygon $\mathcal{P}$ with infinitely many sides. More precisely, the edges and vertices of $\mathcal{P}$ are given as follows: 
\begin{itemize}
	\item[$\ast$] For each $n\geq 0$, let $\gamma_{n}^{+}$ and $\gamma_{n}^{-}$ be the sides of $\mathcal{P}$ coming from cutting $S$ along $\gamma_n$.
	\item[$\ast$] For each $n\geq 0$, let $s_n^+$ and $s_n^-$ be the ideal vertices of $\mathcal{P}$ coming from the puncture $s_n$.
\end{itemize}

\begin{remark}\label{Rem:PolygonWithIdentifications}
	For each $n\geq 0$, the sides $\gamma_{n}^{+}$ and $\gamma_{n}^{-}$ of the ideal hyperbolic polygon $\mathcal{P}$ are identified by an element $g_{n}$ of ${\rm PSL}(2,\mathbb{R})$. If we identify the sides $\gamma_{n}^{+}$ and $\gamma_{n}^{-}$ using $g_{n}$, then we recover the zero-twist tight surface $S(\{l_n,0\})_{n\in\mathbb{N}_{0}}$.
\end{remark}

The ideal hyperbolic polygon $\mathcal{P}$ can be thought in the hyperbolic plane $\mathbb{H}^2$ satisfying the following properties (see Figure \ref{Fig:RealizedPolygonP}):
\begin{itemize}
	\item[$\ast$] Its vertices are on the real axis. Up to take a real translation, we can assume that the vertex $\textbf{0}$ of $\mathcal{P}$ is equal to the complex number zero and that the geodesic ray $\beta$ coincides with the imaginary axis. The collection of vertices $(s_{n}^{+})_{n\in\mathbb{N}_{0}}$ defines a strictly increasing sequence of positive real numbers. The collection of vertices $(s_{n}^{-})_{n\in\mathbb{N}_{0}}$ defines a strictly decreasing sequence of negative real numbers. Given that $\mathcal{P}$ has a reflexive symmetry fixing the ray $\beta$, then $s_{n}^{+}$ and $s_{n}^{-}$ are symmetric with respect to the imaginary axis, it means, $s_n^{-}=-s_n^{+}$, for all $n\geq 0$. 
	
	\item[$\ast$] For each $n\in\mathbb{N}_{0}$, the edges $\gamma_{n}^{+}$ and $\gamma_{n}^{-}$ of $\mathcal{P}$ are half-circles having the same radius and endpoints $s_{n-1}^{+}$ and $s_{n}^{+}$; and $s_{n-1}^{-}$ and $s_{n}^{-}$, respectively. Thus, the intersection of any two of those edges is either empty or they meet at the same point in the real line. Moreover, if we choose  $\gamma$ one of these half-circles, then the other half-circles are in the exterior of $\gamma$.
\end{itemize}
\begin{figure}[h!]
\begin{center}	
	\begin{tikzpicture}[baseline=(current bounding box.north)]
	\begin{scope}[scale=0.9]
	\clip (-6,-0.8) rectangle (8,3.5);
	\draw [red, line width=1pt] (0.5,0) arc(0:180:0.5); 
	\draw [red, line width=1pt] (1.5,0) arc(0:180:0.5);
	\draw [red, line width=1pt] (3.5,0) arc(0:180:1);
	\draw [red, line width=1pt] (-0.5,0) arc(0:180:1);
	\draw [red, line width=1pt] (4.5,0) arc(0:180:0.5);
	\draw [red, line width=1pt] (-2.5,0) arc(0:180:0.5);
	\draw [dashed, blue, line width=1pt] (1,0.5) arc(20:160:0.5);
	\draw [dashed, blue, line width=1pt] (2.6,1) arc(20:160:2.2);
	\draw [dashed, blue, line width=1pt] (4.1,0.5) arc(20:160:3.8);
	\node at (1,0.5)  {{\color{red}\tiny{$<$}}}; 
	\node at (2.5,1) {{\color{red}\tiny{$<$}}};
	\node at (4,0.5) {{\color{red}\tiny{$<$}}};
	\node at (0,0.5)  {{\color{red}\tiny{$>$}}}; 
	\node at (-1.5,1)  {{\color{red}\tiny{$>$}}}; 
	\node at (-3,0.5)  {{\color{red}\tiny{$>$}}}; 
	\node at (-4,0.2) {$\ldots$};
	\node at (5,0.2) {$\ldots$};
	\node at (1.2,0.8) {{\small{\color{red}$\gamma_{0}^{+}$}}};
	\node at (2.8,1.3) {{\small{\color{red}$\gamma_{1}^{+}$}}};
	\node at (4.3,0.8) {{\small{\color{red}$\gamma_{2}^{+}$}}};
	\node at (-0.2,0.8) {{\small{\color{red}$\gamma_{0}^{-}$}}};
	\node at (-1.7, 1.3) {\small{\color{red}{$\gamma_{1}^{-}$}}};
	\node at (-3.3,0.8) {\small{\color{red}{$\gamma_{2}^{-}$}}};
	\node at (0.5,1.1){\small{\color{blue}{$\alpha_{0}$}}};
	\node at (0.5,2.65){\small{\color{blue}{$\alpha_{1}$}}};
	\node at (0.5,3.2){\small{\color{blue}{$\alpha_{2}$}}};
	\node at (0.5,-0.3) {\tiny{$\mathbf{0}$}};
	\node at (1.5,-0.3) {\tiny{$s_{0}^{+}$}};
	\node at (3.5,-0.3) {{\tiny$s_{1}^{+}$}};
	\node at (4.5,-0.3) {\tiny{$s_{2}^{+}$}};
	\node at (-0.5,-0.3) {\tiny{$s_{0}^{-}$}};
	\node at (-2.5,-0.3) {\tiny{$s_{1}^{-}$}};
	\node at (-3.5,-0.3) {\tiny{$s_{2}^{-}$}};
	\draw [->](-4,0) -- (5,0);
	\draw[green] (0.5,0) -- (0.5,3.5);
	\node at (0.7, 1.7){\small{\color{green}{$\beta$}}};
	\end{scope}
	\end{tikzpicture} 
	\caption{\emph{Realization of the ideal hyperbolic polygon $\mathcal{P}$ in $\mathbb{H}^2$.}}
	\label{Fig:RealizedPolygonP}
\end{center}	
\end{figure}
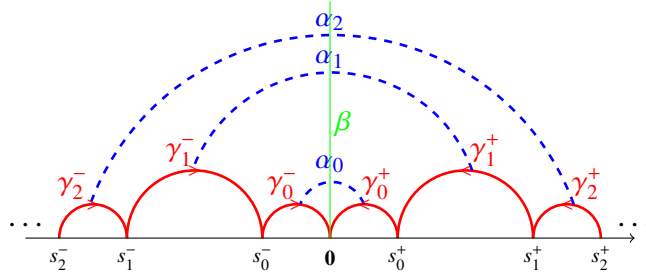

To the zero-twist flute surface $S$ we associated the sequence $\mathbf{x}=(x_{n})_{n\in\mathbb{N}_{0}}$ of positive real numbers that satisfies \begin{equation}
s_{n}^{+}=\sum\limits_{i=0}^{n}x_{i}.
\end{equation}

Thus, to the sequence $\textbf{x}$ we associated the group $\Gamma_{\mathbf{x}}$ given by 
\begin{equation}
\Gamma_{\textbf{x}}=\langle g_{n}: n\in\mathbb{N}_0 \rangle.
\end{equation}

Now we prove that $\Gamma_{\mathbf{x}}$ satisfies the properties described in Theorem \ref{Teo:Parametrization-ZTFS}. 

\noindent For each $n\in\mathbb{N}_{0}$, let $s_n:=s_n^+$.  On the hyperbolic plane $\mathbb{H}^{2}$, we draw the half-circles: $\gamma_{0}^{+}$ and $\gamma_{0}^{-}$ having endpoints $\textbf{0}$ and $s_{0}$, and $\textbf{0}$ and $-s_{0}$, respectively; For each $n\geq 1$, $\gamma_{n}^{+}$ and $\gamma_{n}^{-}$ having endpoints $s_{n-1}$ and $s_{n}$, and $-s_{n-1}$ and $-s_{n}$, respectively. See Figure \ref{Fig:RealizedPolygonP}. We remark that $\gamma_{n}^{+}$ and $\gamma_{n}^{-}$ are symmetric with respect to the imaginary axis. Then the element $g_{n}$ of ${\rm PSL}(2,\mathbb{R})$ is given by
\begin{equation}\label{Ecu:General}
g_n(z):=\dfrac{\left(1+\dfrac{2s_{n-1}}{s_n-s_{n-1}}\right)z-2s_{n-1}\left(1+\dfrac{s_{n-1}}{s_n-s_{n-1}}\right)}{-\dfrac{2}{s_n-s_{n-1}}z+\left(1+\dfrac{2s_{n-1}}{s_n-s_{n-1}}\right)},
\end{equation}
which sends $\gamma_{n}^{+}$ onto $\gamma_{n}^{-}$, for each $n\geq 0$. Observe that $g_{0}$ is parabolic and, for each $n\geq 1$, $g_{n}$ is hyperbolic with trace equal to 
\[
\left\vert 2\left(1+\frac{2s_{n-1}}{s_n-s_{n-1}}\right)\right\vert>2.
\]
As $\Gamma_{\textbf{x}}$ is obtained by side-pairing the sides of the convex ideal polygon $\mathcal{P}$, then it is a non elementary group and its set of generators is composed by non elliptic elements. Thus, $\Gamma_{\textbf{x}}$ is a Fuchsian group (see \cite{Bear1}*{Theorem 8.3.1, p. 198}), with fundamental domain $\mathcal{P}$ for its action on $\mathbb{H}^2$. In what follows we describe $\mathcal{P}$ more precisely.  

Recall that if $\gamma$ is a half-circle in $\mathbb{H}^{2}$ having as center the point $z\in\mathbb{R}$ and radius $r>0$, then $\mathbb{H}^{2} \setminus \gamma$ has two connected components. The connected component of $\mathbb{H}^2\setminus \gamma$  equal to the set $\{w\in\mathbb{H}^{2}:\vert z-w\vert>r\}$ is called the \emph{exterior of $\gamma$} and it is denoted by ${\rm Ext}(\gamma)$. The complement of the closure in $\mathbb{H}^2$ of $\mathrm{Ext}(\gamma)$ is called the \emph{interior of $\gamma$} and is denoted  by ${\rm Int}(\gamma)$, see Figure \ref{Fig:int_ext}.
 \begin{figure}[h!]
 	\begin{center}	
 		\begin{tikzpicture}[baseline=(current bounding box.north)]
 		\begin{scope}[scale=0.8]
 		\clip (-3,-0.8) rectangle (3.7,2.9);
 		\draw [blue, line width=1pt] (2,0) arc(0:180:2);
 		\node at (-2,-0.3) {\tiny{$z-r$}};
 		\node at (0,-0.3) {\tiny{$z$}};
 		\node at (2.1,-0.3) {\tiny{$z+r$}};
 		\node at (0,2.5) {\small{${\rm{Ext}}(\gamma)$}};
 		\draw [->, >=latex, black!30](-3,0) -- (3,0);
 		\node at (0,0.9) {\small{${\rm{Int}}(\gamma)$}};
 		\node at (-1.8,1.5) {\small{{\color{blue}$\gamma$}}};
 		\end{scope}
 		\end{tikzpicture} 
 		\caption{\emph{Exterior and interior of a half-circle $\gamma$} in $\mathbb{H}^{2}$.}
 		\label{Fig:int_ext}
 	\end{center}	
 \end{figure}
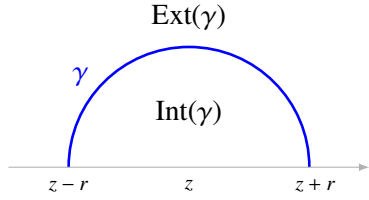

The limit $\lim\limits_{n\to \infty} s_n$ in $\mathbb{R}\cup \{\infty \}$ is well defined because the sequence  $(s_n)_{n\in\mathbb{N}_{0}}$ is the partial sums of $\mathbf{x}$. Let $a$ be such limit. If $a$ is infinity, then $\mathcal{P}$ is equal to the closure (in $\mathbb{H}^2$) of \[D:=\bigcap_{n\in \mathbb{N}_{0}} \left({\rm{Ext}}(\gamma_n^+)\cap {\rm{Ext}}(\gamma_n^-)\right)\subset \mathbb{H}^2.\] Otherwise, if $a$ is a positive real number, then we denote by $\gamma$ the half-circle in $\mathbb{H}^2$ having $a$ and $-a$ as ends points. In this case $\mathcal{P}$ is given by 
\[\overline{D}\cap {\rm{Int}}(\gamma) \subset \mathbb{H}^2.\] 
     	
Notice that $\overline{D}$ is a fundamental domain for $\Gamma_{\mathbf{x}}$ and $\overline{D}/\Gamma_{\mathbf{x}}$ is a complete hyperbolic surface. Therefore, if $a$ is infinity then $\mathcal{P}/\Gamma_{\mathbf{x}}=\overline{D}/\Gamma_{\mathbf{x}}$ is isometric to $S$ and $\Gamma_{\mathbf{x}}$ is a Fuchsian group of the first kind. Otherwise, if $a$ is finite then $\mathcal{P}/\Gamma_{\mathbf{x}}$ is isometric to $S$ and it is not a complete hyperbolic surface because the projection points of $\gamma$ are limit points which are not in the surface. However, $\mathcal{P}/\Gamma_{\mathbf{x}}$ is the convex core of $\overline{D}/\Gamma_{\mathbf{x}}$. Here finished the proof of Theorem \ref{Teo:Parametrization-ZTFS}.\qed



\section{Proof of Theorem \ref{t:LNM}}\label{section:Loch_ness_monster}\label{sec:proof-LNM}
From the introduction we recall that $\mathcal{N}$ denotes the set of all sequences $\mathbf{y}:=(y_n)_{n\in \mathbb{Z}}$ of elements $y_n=(a_{n},b_{n},c_{n},d_{n},e_{n}) \in \mathbb{R}^5$ satisfying 

\begin{equation}\label{Eq:Condition}
    a_{n}<b_{n}<c_{n}<d_{n}<e_{n} \mbox{ and } e_{n}\leq a_{n+1}
\end{equation}

For each $n\in \mathbb{Z}$, let $f_{n}$ and $g_{n}$ be elements of ${\rm PSL}(2,\mathbb{R})$ mapping $\sigma_{n}$ onto $\tilde{\sigma}_{n}$ and $\rho_{n}$ onto $\tilde{\rho}_{n}$, respectively, where $\sigma_{n}$, $\rho_{n}$, $\tilde{\sigma}_{n}$ and $\tilde{\rho}_{n}$ are the half-circles in the hyperbolic plane $\mathbb{H}^2$ depicted in Figure \ref{Fig:half_circleIntro}.  Let 
\begin{equation}
    G_\mathbf{y}:= \langle f_n,g_n:\, n\in \mathbb{Z} \rangle \leq {\rm PSL}(2,\mathbb{R}).
\end{equation}

To start with the proof of Theorem \ref{t:LNM}, the following lemma is required.

\begin{lemma}\label{lemma:hyperbolic_moebius_map}
Let $\sigma$ and $\tilde{\sigma}$ denote the half-circles having endpoints $a$ and $b$; $c$ and $d$ respectively, where $a<b<c<d$ (see Figure \ref{Fig:half_circles_1}). Then the element $f$ of ${\rm PSL}(2,\mathbb{R})$ which sends the half-circle $\sigma$ onto the half-circle $\tilde{\sigma}$ is hyperbolic.
\end{lemma}

\begin{proof}
	By hypothesis, the half-circles $\sigma$ and $\tilde{\sigma}$ are disjoint, and each one of the half-circles is contained in the exterior of the other half-circle.
	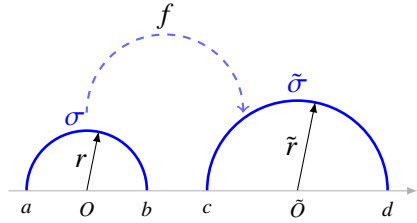
\begin{figure}[h!]
		\begin{center}	
			\begin{tikzpicture}[baseline=(current bounding box.north)]
			\begin{scope}[scale=0.8]
			\clip (-3.3,-0.5) rectangle (3.6,3.3);
			\draw [blue, line width=1pt] (-1,0) arc(0:180:1);
			\draw [blue, line width=1pt] (3,0) arc(0:180:1.5);
			\draw[dashed, color=blue!60, thick, <-] (0.6,1.3) arc (0:180:1.3);
			\node at (-3,-0.3) {\tiny{$a$}};
			\node at (-1,-0.3) {\tiny{$b$}};
			\node at (0,-0.3) {\tiny{$c$}};
			\node at (3,-0.3) {\tiny{$d$}};
			\node at (-2,-0.3) {\tiny{$O$}};
			\node at (-2.1,0.5) {\small{$r$}};
			\draw [->, >=latex](-2,0) -- (-1.8,0.99);
			\node at (1.5,-0.3) {\tiny{$\tilde{O}$}};
			\node at (1.4,0.75) {\small{$\tilde{r}$}};
			\draw [->, >=latex](1.5,0) -- (1.8,1.48);
			\node at (-2.2,1.2) {\small{{\color{blue}$\sigma$}}};
			\node at (1.5,1.8) {\small{{\color{blue}$\tilde{\sigma}$}}};
			\node at (-0.7,2.9) {\small{$f$}};
			\draw [->, >=latex, black!30](-3.3,0) -- (3.5,0);
			\end{scope}
			\end{tikzpicture} 
			\caption{\emph{Half-circles $\sigma$ and $\tilde{\sigma}$.}}
			\label{Fig:half_circles_1}
		\end{center}	
	\end{figure}
	Explicitly, they have centers at the real numbers $O=a+\frac{b-a}{2}$ and $\tilde{O}=c+\frac{d-c}{2}$ respectively, and radius $r=\frac{b-a}{2}$ and $\tilde{r}=\frac{d-c}{2}$, respectively. Then the element $f$ of ${\rm PSL}(2,\mathbb{R})$ is given by  
	\begin{equation}\label{eq:hyperbolic_moebius_map}
	f(z)=\frac{-r\tilde{r}}{z-O} + \tilde{O}.
	\end{equation}
	Let us note that the fixed points of  $f$ are 
	\begin{equation}\label{eq:fix_points}
	z=\frac{-(r+\tilde{r})\pm \sqrt{\delta}}{2},
	\end{equation}
	where $\delta=(O+\tilde{O})^2-4(O\tilde{O}+r\tilde{r})=(O-\tilde{O})^2-4r\tilde{r}$. Since $0<  r+\tilde{r} < \vert \tilde{O}-O \vert$, then 
	$$0\leq (r-\tilde{r})^2= (r+\tilde{r})^2-4r\tilde{r} <(\tilde{O}-O)^2-4r\tilde{r}=\delta.$$
	It implies that $\sqrt{\delta}$ is a real number. Therefore, $f$ is hyperbolic.
\end{proof}

\medskip

\noindent \emph{The group $G_{\mathbf{y}}$  is Fuchsian}. By Lemma \ref{lemma:hyperbolic_moebius_map}, $G_{\mathbf{y}}$ is generated by hyperbolic elements. So, by Theorem 8.2.1 in \cite{Bear1}, $G_{\mathbf{y}}$ is a Fuchsian group. By construction, the collection of half-circles  $\mathcal{C}=\{\sigma_{n}, \tilde{\sigma}_{n}, \rho_{n} ,\tilde{\rho}_{n}:n\in\mathbb{Z}\}$ are pairwise disjoint, and each one of the half-circles in $\mathcal{C}$ is contained in the exterior of each one of the other half-circles in $\mathcal{C}$. From Theorem 3.3.5 in \cite{KS}, a fundamental region $D(G_{\mathbf{y}})$ for the Fuchsian group $G_{\mathbf{y}}$ is given by 

\begin{equation}\label{eq:fundamental_domain_LNM}
D(G_{\mathbf{y}})=\bigcap\limits_{n\in\mathbb{Z}}\left({\rm Ext}(\sigma_{n})\cap {\rm Ext}(\tilde{\sigma}_{n})\cap {\rm Ext}(\rho_{n})\cap {\rm Ext}(\tilde{\rho}_{n})\right)\subset \mathbb{H}^{2}.
\end{equation} 

Given that the intersection of any two different elements belonged to $\mathcal{C}$ is either empty or at infinity, the last means, they meet at the same point in the real line, then $G_{\mathbf{y}}$ acts freely and properly discontinuously on whole $\mathbb{H}^{2}$, see \cite{MB}*{p. 29}. Hence, the quotient space $S:=\mathbb{H}^{2}/G_{\mathbf{y}}$ is a complete Riemann surface.\qed 

\medskip

\noindent In order to prove that $S$ is topologically equivalent to the Loch Ness monster, we must verify that $S$ has only one end and infinite genus. The proof uses the same ideas appearing in \cite{AyC}. 

\medskip 

\noindent \emph{The surface $S$ has only one end}. By Lemma \ref{lemma:spec} is enough to prove that for any compact subset $K$ of $S$, there exists a compact subset $K'$ of $S$ such that $K\subset K'$ and the space $S\setminus K'$ is connected.     	
	
Let $K$ be a compact subset of $S$. Observe that $\tilde{K}:=\pi^{-1}(K)$ is a compact subset of $\overline{D(G_{\mathbf{y}})}$ where $\pi:\overline{D(G_{\mathbf{y}})}\rightarrow S= \overline{D(G_{\mathbf{y}})}/G_{\mathbf{y}}$ is the quotient map. Thus, $\tilde{K}$ is a compact subset of $\mathbb{H}^2$. Therefore, there exist closed intervals $I_x$ and $I_y$  in the real and imaginary axes, respectively, such that, $\pi_x(\tilde{K}) \subset I_x$ and $\pi_y(\tilde{K}) \subset I_y$, where $\pi_{x}:\mathbb{H}^{2}\to\mathbb{R}$ is the standard projection map on the real axis and $\pi_{y}:\mathbb{H}^{2}\to (0,+\infty)$ is the projection map on the imaginary axis. By construction of the fundamental domain for $G_{\mathbf{y}}$ we have that $\tilde{K}':=\overline{D(\Gamma)}\cap I_{x}\times I_{y}$ is a compact subset of $\overline{D(G_{\mathbf{y}})}$ such that $\tilde{K}\subset \tilde{K}'$. Thus, $K':=\pi(\tilde{K}')$ is a compact subset of $S$ which contains $K$. Finally, we can verify that $S\smallsetminus K'$ is connected.

\noindent \emph{The surface $S$ has infinite genus}. For each $n\in\mathbb{Z}$, we define the strip \[ \tilde{S}_{n}:=D(G_{\mathbf{y}})\cap \{z\in\mathbb{H}^2: a_{n}<{\rm Re}(z)<e_{n} \}\subset \mathbb{H}^2.\]

\noindent Observe that $\tilde{S}_{n}/\langle f_n,g_n \rangle$ defines, under the projection map $\pi$, a subsurface $S_{n}$ of $S$ topologically equivalent to a torus minus a disk. By construction, any for two different integers $m\neq n$, the subsurfaces $S_{n}$ and $S_{m}$ are disjoint because the strips $\tilde{S}_n \cap \tilde{S}_m =\emptyset$. This shows that $S$ has infinite genus.

\medskip

\noindent \emph{The group $G_\mathbf{y}$ is of the first kind if and only if $e_n=a_{n+1}$ for all $n\in \mathbb{Z}$, $\lim\limits_{n\to \infty} e_n=\infty$ and $\lim\limits_{n\to -\infty} a_n=-\infty$.} Suppose that $G_{\mathbf{y}}$ if of the first kind but that the conclusion is false. In any case it is possible to find an interval $I\subset \mathbb{R}$ not contained in the limit set of $G_{\mathbf{y}}$ which leads to a contradiction. Indeed, if $e_n\neq a_{n+1}$ for some $n\in \mathbb{Z}$ then $I=(e_n,a_{n+1})$ is such interval. If $\lim\limits_{n\to \infty} e_n=r<\infty$ then $I=(a,\infty)$. Finally, if $\lim\limits_{n\to -\infty} a_n=R> -\infty$, then $I=(-\infty,R)$.  

Now, suppose that $e_n=a_{n+1}$ for all $n\in \mathbb{Z}$, $\lim\limits_{n\to \infty} e_n=\infty$ and $\lim\limits_{n\to -\infty} a_n=-\infty$ but that $G_{\mathbf{y}}$ is not of the first kind. Let $x$ be a point in $\partial \mathbb{H}^2$ such that it is not in  the limit set $\Lambda(G_{\mathbf{y}})$ of $G_{\mathbf{y}}$. Then there exists an open interval $L \subset \partial \mathbb{H}^2$ with $x\in L$ and $\Lambda(G_{\mathbf{y}})\cap L =\emptyset$. Let $C$ be the half-circle in $\mathbb{H}^2$ with endpoints the endpoints of $L$. Given that $\mathbb{H}^{2}=\bigcup\limits_{g\in G_{\mathbf{y}}} g(\overline{D(G_{\textbf{y}})})$ then there exists $g\in G_{\mathbf{y}}$ such that $C\cap g(D(G_{\mathbf{y}}))\neq \emptyset$. Therefore, $g^{-1}(C)\cap D(G_{\mathbf{y}})\neq \emptyset$. By hypothesis we have that $g^{-1}(L)\cap \Lambda(G_{\mathbf{y}})\neq \emptyset$. Given that the limit set is $G_{\mathbf{y}}$-invariant, then we obtain that $L\cap \Lambda(G_{\mathbf{y}})\neq \emptyset$, it is a contradiction because we suppose that $\Lambda(G_{\mathbf{y}})\cap L =\emptyset$. \qed

\vspace{4mm}

From the point of view of parabolicity the following questions appears naturally: Which points $\mathbf{y}$ of $\mathcal{N}$ define Fuchsian groups uniformizing Loch Ness monsters of parabolic type?


\section*{Acknowledgements}

We thank Jos\'e Hern\'andez Santiago, Rub\'en Hidalgo and Jes\'us Muci\~no for his effort to review our paper and for the provided us with great suggestions and comments.

I. Morales was partially supported by \emph{``Proyecto CONACyT Ciencia de Frontera 217392 cerrando brechas y extendiendo puentes en Geometr\'ia y Topolog\'ia''}. 

Camilo Ram\'irez Maluendas expresses his gratitude to Universidad Nacional de Colombia sede Manizales, and to the research program \emph{``Reconstrucci\'on del tejido social en zonas posconflicto en Colombia''} SIGP code: 57579 with the project entitled: \emph{``Competencias empresariales y de innovaci\'on para el desarrollo econ\'omico y la inclusi\'on productiva de las regiones afectadas por el conflicto colombiano''} SIGP code 58907. Contract number: FP44842-213-2018. He dedicates this work to his beautiful family, Marbella and Emilio, in gratitude for their love and support.



\end{document}